\documentclass[12pt, reqno]{amsart}
\usepackage{amsmath, amsthm, amscd, amsfonts, amssymb, graphicx, color}
\usepackage[bookmarksnumbered, colorlinks, plainpages]{hyperref}
\hypersetup{colorlinks=true,linkcolor=red, anchorcolor=green, citecolor=cyan, urlcolor=red, filecolor=magenta, pdftoolbar=true}

\newtheorem{theorem}{Theorem}[section]
\newtheorem{lemma}[theorem]{Lemma}
\newtheorem{proposition}[theorem]{Proposition}
\newtheorem{corollary}[theorem]{Corollary}
\theoremstyle{definition}

\theoremstyle{remark}
\newtheorem{remark}[theorem]{Remark}
\numberwithin{equation}{section}

\begin{document}
\setcounter{page}{1}

\title[Geometric aspects of $p$-angular and skew $p$-angular distances]{Geometric aspects of $p$-angular\\ and skew $p$-angular distances}

\author[J. Rooin, S. Habibzadeh, M.S. Moslehian]{J. Rooin$^{1}$, S. Habibzadeh$^{2}$ and M.S. Moslehian$^{3}$}

\address{$^{1}$Department of Mathematics, Institute for Advanced Studies in Basic Sciences (IASBS), Zanjan 45137-66731, Iran}
\email{rooin@iasbs.ac.ir}

\address{$^{2}$Department of Mathematics, Institute for Advanced Studies in Basic Sciences (IASBS), Zanjan 45137-66731, Iran;\newline The Tusi Mathematical Research Group (TMRG), Mashhad, Iran.}
\email{s.habibzadeh@iasbs.ac.ir}

\address{$^{3}$ Department of Pure Mathematics, Center of Excellence in
Analysis on Algebraic Structures (CEAAS), Ferdowsi University of
Mashhad, P. O. Box 1159, Mashhad 91775, Iran.}
\email{moslehian@um.ac.ir, moslehian@member.ams.org}



\begin{abstract}
Corresponding to the concept of $p$-angular distance $\alpha_p[x,y]:=\left\lVert\lVert x\rVert^{p-1}x-\lVert y\rVert^{p-1}y\right\rVert$, we first introduce the notion of skew $p$-angular distance $\beta_p[x,y]:=\left\lVert \lVert y\rVert^{p-1}x-\lVert x\rVert^{p-1}y\right\rVert$ for non-zero elements of $x, y$ in a real normed linear space and study some of significant geometric properties of the $p$-angular and the skew $p$-angular distances. We then give some results comparing two different $p$-angular distances with each other. Finally, we present some characterizations of inner product spaces related to the $p$-angular and the skew $p$-angular distances. In particular, we show that if $p>1$ is a real number, then a real normed space $\mathcal{X}$ is an inner product space, if and only if for any $x,y\in \mathcal{X}\smallsetminus{\lbrace 0\rbrace}$, it holds that $\alpha_p[x,y]\geq\beta_p[x,y]$.
\end{abstract}

\maketitle

\section{Introduction}

Throughout this paper, let $\mathcal{X}$ denotes an arbitrary non-zero normed linear space over the field of real numbers.\par
Clarkson \cite{C} introduced  the concept of angular distance between non-zero elements $x$ and $y$ in $\mathcal{X}$ by
\begin{equation*}
\alpha[x,y]=\left\Vert\frac{x}{\Vert x\Vert}-\frac{y}{\Vert y\Vert}\right\Vert.
\end{equation*}
In \cite{M}, Maligranda considered the $p$-angular distance
\begin{equation*}
\alpha_{p}[x,y]=\left\Vert\frac{x}{\Vert x\Vert^{1-p}}-\frac{y}{\Vert y\Vert^{1-p}}\right\Vert\qquad(p\in\mathbb{R})
\end{equation*}
between non-zero vectors $x$ and $y$ in $\mathcal{X}$ as a generalization of the concept of angular distance.
Corresponding to the notion of $p$-angular distance, we define the concept of skew $p$-angular distance between non-zero vectors $x$ and $y$ in $\mathcal{X}$, as
\begin{equation*}
\beta_{p}[x,y]=\bigg\Vert\frac{x}{\Vert y\Vert^{1-p}}-\frac{y}{\Vert x\Vert^{1-p}}\bigg\Vert\qquad(p\in\mathbb{R}).
\end{equation*}
We set $\beta[x,y]$ for $\beta_{p}[x,y]$ when $p=0$ and call it skew angular distance between non-zero elements $x$ and $y$ in $\mathcal{X}$. Evidently, it holds that
\begin{equation}\label{pp}
\beta_p[x,y]=\lVert x\rVert^{p-1}\lVert y\rVert^{p-1}\alpha_{2-p}[x,y].
\end{equation}\par
Dunkl and Williams \cite{D9} obtained a useful upper bound for the angular distance. They showed that
\begin{equation*}
\alpha[x,y]\leq\frac{4\lVert x-y\rVert}{\lVert x\rVert+\lVert y\rVert}.
\end{equation*}
The following result providing a lower bound for the $p$-angular distance was
stated without a proof by Gurari\u{l} in \cite{VI}:
\begin{equation*}
2^{-p}\lVert x-y\rVert^{p}\leq\alpha_p[x,y],
\end{equation*}
where $p\geq1$ and $x,y\in \mathcal{X}$.\par
Finally, we recall the result of Hile \cite{khu}:
\begin{equation}\label{nbn}
\alpha_p[x,y]\leq\frac{\lVert y\rVert^{p}-\lVert x\rVert^{p}}{\lVert y\rVert-\lVert x\rVert}\lVert x-y\rVert,
\end{equation}
for $p\geq1$ and $x,y\in \mathcal{X}$ with $\lVert x\rVert\neq\lVert y\rVert$.
For some recently obtained upper and lower bounds for the $p$-angular distance the reader is referred to \cite{D2, D1} and \cite{M}.\par
Numerous basic characterizations of inner product spaces under various conditions were first given by Fr\'{e}chet, Jordan and von Neumann; see \cite{DFM} and references therein. Since then,
the problem of finding necessary and sufficient conditions for a normed space to be an inner product space has been
investigated by many mathematicians by considering some types of orthogonality or some geometric aspects of underlying
spaces; see, e.g., \cite{F, lorch}. There is an interesting book by Amir \cite{A} that contains several characterizations of inner product spaces, which are based on norm inequalities, various notions of orthogonality in normed linear spaces and so on.
Among significant characterizations of inner product spaces related to $p$-angular distance, we can mention \cite{Al, DFM, DSS, H}. The next two theorems due to Lorch and Ficken will be used in this paper.

\noindent
$\mathbf{\textbf{Theorem A}} (\mathcal{\textbf{Lorch}})$ \cite{lorch}. \textit{Let $(\mathcal{X},\Vert\cdot\Vert)$ be a normed space. Then the following statements are mutually equivalent:
\begin{itemize}
\item [\rm{(i)}]~ For each $x,y\in \mathcal{X}$ if $\Vert x\Vert=\Vert y\Vert $, then $\Vert x+y\Vert\leq \Vert\lambda x+\lambda^{-1}y\Vert$ (for all $\lambda\neq 0$).
\item [\rm{(ii)}]~ For each $x,y\in \mathcal{X}$ if $\Vert x+y\Vert\leq\Vert\lambda x+\lambda^{-1}y\Vert$ (for all $\lambda\neq 0$), then $\Vert x\Vert=\Vert y\Vert$.
\item [\rm{(iii)}]~ $(\mathcal{X},\Vert\cdot\Vert)$ is an inner product space.
\end{itemize}}
\noindent

$\mathbf{\textbf{Theorem B }} (\mathcal{\textbf{Ficken}})$ \cite{F}. \textit{Let $(\mathcal{X},\Vert\cdot\Vert)$ be a normed space. Then the following statements are mutually equivalent:
\begin{itemize}
\item [\rm{(i)}]~ For each $x, y\in \mathcal{X}$ if $\lVert x\rVert=\lVert y\rVert$, then $\lVert \alpha x+\beta y\rVert=\lVert \beta x+\alpha y\rVert$ (for all $\alpha,\beta>0$).
\item [\rm{(ii)}]~ For each $x, y\in \mathcal{X}$ if $\lVert x\rVert=\lVert y\rVert$, then $\lVert\lambda x+\lambda^{-1}y\rVert=\lVert\lambda^{-1}x+\lambda y\rVert$ for all $\lambda>0$.
\item [\rm{(iii)}]~ $(\mathcal{X},\Vert\cdot\Vert)$ is an inner product space.
\end{itemize}}
In this paper, first we study some topological aspects of $p$-angular distances such as metrizability, consistency and completeness. Then, we compare two arbitrary $p$-angular and $q$-angular distances with each other, which generalize the results of Maligranda \cite{M} and Dragomir \cite{D2}. Finally, we present two different characterizations of inner product spaces related to the $p$-angular and the skew $p$-angular distances.

\section{Some initial observations}
In this section, first we examine some topological facts of the $p$-angular and the skew $p$-angular distances. Then we compare the $p$-angular distance with the skew $p$-angular distance in inner product spaces and give suitable representations for the $p$-angular distance, which will be used in the sequel for characterizations of inner product spaces.

\subsection{Geometric properties of the $p$-angular distance}

In this subsection, we study the metrizability, the consistency and the completeness concepts regarding to the $p$-angular and the skew $p$-angular distances.
\begin{theorem}\label{hgi}
For $p\neq0$, $\alpha_p[x,y]$ is a metric on $\mathcal{X}\smallsetminus{\lbrace 0\rbrace}$, which is consistent with $\alpha_1[x,y]=\lVert x-y\rVert$; they induce the same topology on $\mathcal{X}\smallsetminus{\lbrace 0\rbrace}$. If $p$ and $q$ are distinct non-zero real numbers, then $\alpha_p$ is not equivalent with $\alpha_q$. If $p\neq 1$, then $\alpha_p$ is not translation invariant.
\end{theorem}
\begin{proof}
Clearly $\alpha_p$ is a metric. Let $\alpha_1[x_n, x]=\lVert x_n-x\rVert\rightarrow 0$ as $n\rightarrow\infty$ in $\mathcal{X}\smallsetminus{\lbrace 0\rbrace}$. Thus $\lim_{n\rightarrow\infty}\lVert x_n\rVert=\lVert x\rVert$, and so
\small\begin{align*}
\alpha_p[x_n, x]&=\left\lVert\lVert x_n\rVert^{p-1}x_n-\lVert x\rVert^{p-1}x\right\rVert\\
&\leq\left\lVert\lVert x_n\rVert^{p-1}x_n-\lVert x_n\rVert^{p-1}x\right\rVert+\left\lVert\lVert x_n\rVert^{p-1}x-\lVert x\rVert^{p-1}x\right\rVert\\
&\leq\lVert x_n\rVert^{p-1}\lVert x_n-x\rVert+\lVert x\rVert\left\lvert\lVert x_n\rVert^{p-1}-\lVert x\rVert^{p-1}\right\rvert\rightarrow 0\quad(\text{as $n\rightarrow\infty$}).
\end{align*}
Therefore the topology of $\alpha_p$ is weaker than the topology of $\alpha_1$ on $\mathcal{X}\smallsetminus{\lbrace 0\rbrace}$.\par
Now we assume that $\alpha_p[x_n, x]\rightarrow 0$ as $n\rightarrow\infty$ in $\mathcal{X}\smallsetminus{\lbrace 0\rbrace}$. We have
\small\begin{equation*}
\bigg\lvert\frac{\lVert x_n\rVert}{\lVert x_n\rVert^{1-p}}-\frac{\lVert x\rVert}{\lVert x\rVert^{1-p}}\bigg\rvert\leq\bigg\lVert\frac{x_n}{\lVert x_n\rVert^{1-p}}-\frac{x}{\lVert x\rVert^{1-p}}\bigg\rVert\rightarrow 0\quad(\text{as $n\rightarrow\infty$}),
\end{equation*}
and so $\lim_{n\rightarrow\infty}\lVert x_n\rVert^{p}=\lVert x\rVert^{p}$, which implies that $\lim_{n\rightarrow\infty}\lVert x_n\rVert=\lVert x\rVert$. Thus,
\small\begin{align*}
\alpha_1[x_n, x]&=\lVert x_n-x\rVert=\lVert x\rVert^{1-p}\left\lVert\lVert x\rVert^{p-1}x_n-\lVert x\rVert^{p-1}x\right\rVert\\
&\leq\lVert x\rVert^{1-p}(\left\lVert\lVert x\rVert^{p-1}x_n-\lVert x_n\rVert^{p-1}x_n\right\rVert+\left\lVert\lVert x_n\rVert^{p-1}x_n-\lVert x\rVert^{p-1}x\right\rVert)\\
&=\lVert x\rVert^{1-p}(\lVert x_n\rVert\left\lvert\lVert x\rVert^{p-1}-\lVert x_n\rVert^{p-1}\right\rvert+\alpha_p[x_n, x])\rightarrow 0\quad(\text{as $n\rightarrow\infty$}).
\end{align*}
Therefore the topology of $\alpha_1$ is weaker than the topology of $\alpha_p$ on $\mathcal{X}\smallsetminus{\lbrace 0\rbrace}$. Hence these two metrics are consistent on $\mathcal{X}\smallsetminus{\lbrace 0\rbrace}$.
Next, let $p,q\in\mathbb{R}\smallsetminus{\lbrace 0\rbrace}$ such that $p<q$. By contrary, assume that there exists a number $M>0$ such that for every $x, y \in \mathcal{X}\smallsetminus{\lbrace 0\rbrace}$, $\alpha_p[x,y]\leq M \alpha_q[x,y]$. Fix a unit vector $a\in \mathcal{X}\smallsetminus{\lbrace 0\rbrace}$. For each  $\lambda, \mu>0$, we have $\alpha_p[\lambda a,\mu a]\leq M \alpha_q[\lambda a,\mu a]$, or
$\lvert \lambda^{p}-\mu^{p}\rvert\leq M\lvert \lambda^{q}-\mu^{q}\rvert$. In particular, if we put $\lambda=\frac{1}{n}$ and $\mu=\frac{t}{n}$ where $n=1,2,\ldots$ and $t>0$, then we have $n^{q-p}\lvert 1-t^{p}\rvert\leq M\lvert 1-t^{q}\rvert$, or $M\geq n^{q-p}\big\lvert\frac{1-t^{p}}{1-t^{q}}\big\rvert~(t\neq1)$. Now letting $t\rightarrow\infty$ in the case when $p<q<0$, and $t\rightarrow 0$ in the case when $0<p<q$, we get $M\geq n^{q-p}~(n=1,2,\ldots)$, and so $M=\infty$, which is a contradiction. In the case where $p<0<q$ taking $\mu=1$, we obtain $\lvert \lambda^{p}-1\rvert\leq M\lvert \lambda^{q}-1\rvert$. Now letting $\lambda\rightarrow 0^{+}$, we get $M=\infty$, a contradiction. Therefore $\alpha_p$ is not equivalent to $\alpha_q$.\par
Now, we show that if $p\neq 1$, then $\alpha_p$ is not translation invariant. By contrary, assume that for each $x,y,z\in \mathcal{X}$ we have $\alpha_p[x+z,y+z]=\alpha_p[x,y]$, whenever $ x, y, x+z, y+z\neq0$. Fixing a unit vector $a\in \mathcal{X}\smallsetminus{\lbrace 0\rbrace}$, put $x=\lambda a$, $y=\gamma a$ and $z=\mu a$, where $\lambda, \mu, \gamma\in\mathbb{R}$. In particular, if we put $\lambda=\mu=1$ and $\gamma>0$, then we have $\lvert 2^{p}-(\gamma+1)^{p}\rvert=\lvert 1-\gamma^{p}\rvert$. Now letting $\gamma\rightarrow\infty$ in the case where $p<0$, and $\gamma\rightarrow0$ in the case where $p>0$, we get a contradiction. In the case where $p=0$, we also get a contradiction by taking $\lambda=1, \mu=-2$ and $\gamma=-1$.
This completes the proof.
\end{proof}
\begin{remark}
It may happen that two metrics $d_1$ and $d_2$ on a set $\mathcal{E}$ are consistent and there exists $m>0$ such that $md_2\leq d_1$ but there exists no $M>0$ such that $d_1\leq M d_2$. For a classical example, take $\mathcal{E}=[1,\infty), d_1(x,y)=\lvert x-y\rvert$ and $d_2(x,y)=\big\lvert\frac{1}{x}-\frac{1}{y}\big\rvert$. Two metrics $d_1$ and $d_2$ induce the same topology on $\mathcal{E}$ and $d_2(x,y)\leq d_1(x,y)$, but since $d_2$ is bounded, there exists no $M>0$ such that $d_1\leq Md_2$.
Since in Theorem \ref{hgi}, $p$ and $q$ are arbitrary, this case cannot occur.
\end{remark}
In spite of $\alpha_p$, the following remark shows that when $p\neq1$, never $\beta_p$ is a metric on $\mathcal{X}\smallsetminus{\lbrace 0\rbrace}$.
\begin{remark}
Let $\mathcal{X}$ be a normed linear space. Take $a\in \mathcal{X}$ with $\lVert a\rVert=1$, and put $x=ra, y=sa$, $z=ta$, where $r,s,t\in\mathbb{R}$. Let $p>1$ and take $r=1$, $s=-1$ and $t>0$. We obtain
\begin{equation*}
\beta_p[x,y]=2>\lvert t^{p-1}-t\rvert+\lvert t^{p-1}+t\rvert
=\beta_p[x,z]+\beta_p[y,z],
\end{equation*}
for small enough t.
This shows that $\beta_p$ is not a metric on $\mathcal{X}\smallsetminus{\lbrace 0\rbrace}$ in this case.
Now let $p<1$, and take $r=2$, $t=1$ and $s>0$. Since for small enough $s$,
\begin{align*}
(2s)^{1-p}\beta_p[x,y] & =\lvert 2^{2-p}-s^{2-p}\rvert
>\lvert s^{1-p}(2^{2-p}-1)\rvert+\lvert2^{1-p}(s^{2-p}-1)\rvert\\
& =(2s)^{1-p}\beta_p[x,z]+(2s)^{1-p}\beta_p[y,z],
\end{align*}
$\beta_p$ is not a metric on $\mathcal{X}\smallsetminus{\lbrace 0\rbrace}$.
\end{remark}
Now we are going to compare completeness of an arbitrary nonempty subset of $\mathcal{X}\smallsetminus{\lbrace 0\rbrace}$ with respect to $\alpha_p$ and $\alpha_q$. To do this, we need some lemmas.
\begin{lemma}
Let $p\neq0$, $A$ be a nonempty $\alpha_p$-complete subset of $\mathcal{X}\smallsetminus{\lbrace 0\rbrace}$ and $\{x_n\}$ be a Cauchy sequence in $\mathcal{X}\smallsetminus{\lbrace 0\rbrace}$. Then
\begin{itemize}
\item [\rm{(i)}]~ If $p>0$, then $A$ is norm-bounded from below, and if $p<0$, then $A$ is norm-bounded from above.
\item [\rm{(ii)}]~ If $p>0$, then $\{x_n\}$ is norm-bounded from above, and if $p<0$, then $\{x_n\}$ is norm-bounded from below.
\end{itemize}
\end{lemma}
\begin{proof}
$(i)$ Let $p>0$. By contrary, assume that there exists a sequence $\{x_n\}$ in $A$ such that $\lim_{n\rightarrow\infty}\lVert x_n\rVert=0$. Therefore
\small\begin{equation*}
\alpha_p[x_m,x_n]=\Big\lVert\frac{x_m}{\lVert x_m\rVert^{1-p}}-\frac{x_n}{\lVert x_n\rVert^{1-p}}\Big\rVert\leq\lVert x_m\rVert^{p}+\lVert x_n\rVert^{p}\rightarrow 0\qquad(m,n\rightarrow\infty),
\end{equation*}
and so $\{x_n\}$ is a $\alpha_p$-Cauchy sequence. Since $A$ is $\alpha_p$-complete, there exists $x\in A$ such that $\lim_{n\rightarrow\infty}\alpha_p[x_n,x]=0$. Since $\left\lvert\lVert x_n\rVert^{p}-\lVert x\rVert^{p}\right\rvert\leq \alpha_p[x_n,x]$,
we get $\lVert x\rVert^{p}=\lim_{n\rightarrow\infty}\left\lvert\lVert x_n\rVert^{p}-\lVert x\rVert^{p}\right\rvert\leq0$, and so $x=0$, which is a contradiction. Therefore, $A$ is norm-bounded from below.
\par
Now, let $p<0$. If $A$ is  not norm-bounded from above, then there exists a sequence $\{x_n\}$ in $A$ such that $\lim_{n\rightarrow\infty}\lVert x_n\rVert=\infty$. By a similar argument we conclude that $\{x_n\}$ is a $\alpha_p$-Cauchy sequence in $A$ and so there exists $x\in A$ such that $\lim_{n\rightarrow\infty}\alpha_p[x_n,x]=0$. Therefore $\lVert x\rVert^{p}=0$, which is impossible.\par
$(ii)$ Obvious.
\end{proof}
The following lemma comparing $\alpha_p$ with $\alpha_q$ without any restrictions on $p$ and $q$, plays an essential role in our study.
\begin{lemma}\label{ppoo}
Let $p,q\in\mathbb{R}$ and $q\neq0$. Then for any non-zero elements $x,y\in \mathcal{X}$,
\begin{align}\label{1}
\frac{\lvert p\rvert}{\lvert p\rvert+\lvert p-q\rvert}&\min(\lVert x\rVert^{p-q},\lVert y\rVert^{p-q})\alpha_{q}[x,y]\nonumber\\
&\leq\alpha_p[x,y]\\
&\leq\frac{\lvert q\rvert+\lvert p-q\rvert}{\lvert q\rvert}\max(\lVert x\rVert^{p-q},\lVert y\rVert^{p-q})\alpha_{q}[x,y].\nonumber
\end{align}
In particular if $q=1$, then
\begin{align}\label{2}
\frac{\lvert p\rvert}{\lvert p\rvert+\lvert p-1\rvert}&\min(\lVert x\rVert^{p-1},\lVert y\rVert^{p-1})\lVert x-y\rVert\nonumber\\
&\leq\alpha_p[x,y]\\
&\leq(1+\lvert p-1\rvert)\max(\lVert x\rVert^{p-1},\lVert y\rVert^{p-1})\lVert x-y\rVert.\nonumber
\end{align}
\end{lemma}
\begin{proof}
We have
\begin{align*}
\alpha_p[x,y]&=\left\lVert\lVert x\rVert^{p-1}x-\lVert y\rVert^{p-1}y\right\rVert\\
&\leq\left\lVert\lVert x\rVert^{p-q}\lVert x\rVert^{q-1}x-\lVert x\rVert^{p-q}\lVert y\rVert^{q-1}y\right\rVert\\
&+\left\lVert\lVert x\rVert^{p-q}\lVert y\rVert^{q-1}y-\lVert y\rVert^{p-q}\lVert y\rVert^{q-1}y\right\rVert\\
&=\lVert x\rVert^{p-q}\alpha_{q}[x,y]+\lVert y\rVert^{q}\left\lvert\lVert x\rVert^{p-q}-\lVert y\rVert^{p-q}\right\rvert.
\end{align*}
Consider the function $f(t)=t^{\frac{p-q}{q}}$ on the closed interval with endpoints $\lVert x\rVert^{q}$ and $\lVert y\rVert^{q}$. By the Mean-Value Theorem, there exists a point $\eta$ between $\lVert x\rVert^{q}$ and $\lVert y\rVert^{q}$ such that
\begin{equation*}
\left\lvert\lVert x\rVert^{p-q}-\lVert y\rVert^{p-q}\right\rvert=\lvert f(\lVert x\rVert^{q})-f(\lVert y\rVert^{q})\rvert
=\big\lvert\frac{p-q}{q}\big\rvert\eta^{\frac{p-2q}{q}}\left\lvert\lVert x\rVert^{q}-\lVert y\rVert^{q}\right\rvert.
\end{equation*}
Since the function $t^{\frac{p-2q}{q}}$ is monotone, we obtain
\begin{equation*}
\eta^{\frac{p-2q}{q}}\leq\max(\lVert x\rVert^{p-2q},\lVert y\rVert^{p-2q}),
\end{equation*}
whence
\small\begin{align*}
\alpha_p[x,y]&\leq\lVert x\rVert^{p-q}\alpha_{q}[x,y]+\big\lvert\frac{p-q}{q}\big\rvert\max(\lVert x\rVert^{p-2q}\lVert y\rVert^{q},\lVert y\rVert^{p-q})\left\lvert\lVert x\rVert^{q}-\lVert y\rVert^{q}\right\rvert\\
&\leq\big(\lVert x\rVert^{p-q}+\big\lvert\frac{p-q}{q}\big\rvert\max(\lVert x\rVert^{p-2q}\lVert y\rVert^{q},\lVert y\rVert^{p-q})\big)\alpha_q[x,y].
\end{align*}
Thus,
\begin{equation}\label{89}
\alpha_p[x,y]\leq\frac{\lvert q\rvert+\lvert p-q\rvert}{\lvert q\rvert}\max(\lVert x\rVert^{p-q},\lVert x\rVert^{p-2q}\lVert y\rVert^{q},\lVert y\rVert^{p-q})\alpha_q[x,y].
\end{equation}
By symmetry, we have
\begin{equation}\label{99}
\alpha_p[x,y]\leq\frac{\lvert q\rvert+\lvert p-q\rvert}{\lvert q\rvert}\max(\lVert y\rVert^{p-q},\lVert y\rVert^{p-2q}\lVert x\rVert^{q},\lVert x\rVert^{p-q})\alpha_q[x,y].
\end{equation}
For proving $(\ref{1})$ we can assume that $\lVert x\rVert\leq\lVert y\rVert$. If $q<0$, then $\lVert y\rVert^{q}\leq\lVert x\rVert^{q}$ and so $\lVert x\rVert^{p-2q}\lVert y\rVert^{q}\leq\lVert x\rVert^{p-q}$. Now, the right inequality in $(\ref{1})$ follows from $(\ref{89})$. Similarly if $q>0$, then $(\ref{99})$ yields the right inequality in $(\ref{1})$.
The left inequality in $(\ref{1})$ follows from the right one by interchanging the roles of $p$ and $q$.
\end{proof}
\begin{theorem}
The following statements hold.
\begin{itemize}
\item [\rm{(i)}]~ If $pq>0$, then for each $\emptyset\neq A\subseteq \mathcal{X}\smallsetminus{\lbrace 0\rbrace}$, the metric space $(A,\alpha_p)$ is complete if and only if $(A,\alpha_q)$ is complete.
\item [\rm{(ii)}]~ If $p>0$ and $q<0$, then there exist nonempty sets $A, B\subseteq \mathcal{X}\smallsetminus{\lbrace 0\rbrace}$ such that $A$ is $\alpha_p$-complete but not $\alpha_q$-complete and $B$ is $\alpha_q$-complete but not $\alpha_p$-complete.
\end{itemize}
\end{theorem}
\begin{proof}
$(i)$
Let $\emptyset\neq A\subseteq \mathcal{X}\smallsetminus{\lbrace 0\rbrace}$ be $\alpha_p$-complete. Assume $\{x_n\}$ is a $\alpha_q$-Cauchy sequence in $A$. 
First, suppose that $p,q>0$. Since $A$ is $\alpha_p$-complete, $A$, and as a result, $\{x_n\}$ is norm-bounded from below. On the other hand, since $\{x_n\}$ is $\alpha_q$-Cauchy, $\{x_n\}$ is norm-bounded from above. Thus, $\{x_n\}$ is norm-bounded from below and above, and so there exists $0\leq M<\infty$ such that $\max(\lVert x_m\rVert^{p-q},\lVert x_n\rVert^{p-q})\leq M~(m,n=1,2,\ldots)$. Therefore by the right hand side of inequality $(\ref{1})$,
\begin{equation*}
\alpha_p[x_m,x_n]\leq\frac{\lvert q\rvert+\lvert p-q\rvert}{\lvert q\rvert}M\alpha_q[x_m,x_n]\rightarrow 0\quad(m,n\rightarrow\infty).
\end{equation*}
Hence, $\{x_n\}$ is a $\alpha_p$-Cauchy sequence in $A$. Since $A$ is $\alpha_p$-complete, there exists $x\in A$ such that $\lim_{n\rightarrow\infty}\alpha_p[x_n,x]=0$, and by the consistency of $\alpha_p$ and $\alpha_q$, we reach $\lim_{n\rightarrow\infty}\alpha_q[x_n,x]=0$. So, $A$ is $\alpha_q$-complete.
\par Now, let $p,q<0$. Since $A$ is $\alpha_p$-complete, $\{x_n\}$ is norm-bounded from above. On the other hand, since $\{x_n\}$ is $\alpha_q$-Cauchy, $\{x_n\}$ is norm-bounded from below. So, $\{x_n\}$ is again norm-bounded from above and below. Similar to the above argument, there exists $x\in A$ such that $\lim_{n\rightarrow\infty}\alpha_q[x_n,x]=0$, and therefore $A$ is $\alpha_q$-complete.\par
$(ii)$ Take a unit vector $a\in \mathcal{X}$ and let $A=\{\lambda a:\lambda\geq1\}$ and $B=\{\lambda a:0<\lambda\leq1\}$. It is easily seen that $A$ is $\alpha_p$-complete. Since $q<0$ and $A$ is not norm-bounded from above, $A$ is not $\alpha_q$-complete. Similarly $B$ is $\alpha_q$-complete, but not $\alpha_p$-complete.
\end{proof}

\subsection{$p$-angular distance in inner product spaces}

In this part, we suppose that $(\mathcal{X}, \langle\cdot,\cdot\rangle)$ is a real inner product space with the induced norm $\Vert \cdot \Vert$, defined by $\lVert x\rVert^{2}=\langle x,x\rangle$.
\begin{proposition}\label{667}
Let $\mathcal{X}$ be an inner product space, $x,y\in \mathcal{X}\smallsetminus{\lbrace 0\rbrace}$ and $p\in\mathbb{R}$. Then the following properties hold.
\begin{itemize}
\item[(i)]   $\alpha_{p}[x,y]\leq\beta_{p}[x,y]$ for all $p<1$,\\
\item[(ii)]  $\alpha_{p}[x,y]=\beta_{p}[x,y]$ for $p=1$,\\
\item[(iii)] $\alpha_{p}[x,y]\geq\beta_{p}[x,y]$ for all $p>1$.
\end{itemize}
In each of $(i)$ and $(iii)$ equality holds if and only if $\Vert x\Vert=\Vert y\Vert$.
\end{proposition}
\begin{proof}
It is sufficient to note that
\begin{equation*}\label{3}
\alpha_{p}[x,y]^{2}-\beta_{p}[x,y]^{2}=(\Vert x\Vert^{2}-\Vert y\Vert^{2})(\Vert x\Vert^{2p-2}-\Vert y\Vert^{2p-2}).
\end{equation*}
\end{proof}
\begin{proposition}\label{867}
Let $\mathcal{X}$ be an inner product space, $p\in\mathbb{R}$ and $x,y\in \mathcal{X}\smallsetminus{\lbrace 0\rbrace}$. Then
\small\begin{equation*}
\alpha_p[x, y]=\sqrt{(\lVert x\rVert^{p+1}-\lVert y\rVert^{p+1})(\lVert x\rVert^{p-1}-\lVert y\rVert^{p-1})+\lVert x\rVert^{p-1}\lVert y\rVert^{p-1}\lVert x-y\rVert^{2}}.
\end{equation*}
In particular if $p=0$, then
\begin{equation*}
\alpha[x,y]=\sqrt{\frac{\lVert x-y\rVert^{2}-(\lVert x\rVert-\lVert y\rVert)^{2}}{\lVert x\rVert\lVert y\rVert}}.
\end{equation*}
\end{proposition}
\begin{proof}
The first identity follows from
\begin{align*}
\alpha_p^2[x,y]&=
\lVert x\rVert^{2p}+\lVert y\rVert^{2p}-2\lVert x\rVert^{p-1}\lVert y\rVert^{p-1}Re\langle x, y\rangle\\
&=\lVert x\rVert^{2p}+\lVert y\rVert^{2p}-\lVert x\rVert^{p-1}\lVert y\rVert^{p-1}(\lVert x\rVert^{2}+\lVert y\rVert^{2}-\lVert x-y\rVert^{2})\\
&=(\lVert x\rVert^{p+1}-\lVert y\rVert^{p+1})(\lVert x\rVert^{p-1}-\lVert y\rVert^{p-1})+\lVert x\rVert^{p-1}\lVert y\rVert^{p-1}\lVert x-y\rVert^{2}.
\end{align*}
\end{proof}
Now, we show the relation between the $p$-angular distance and the errors of the Cauchy-Schwarz inequality; $\lVert x\rVert\lVert y\rVert\pm\langle x,y\rangle\geq0$. For this reason we need the following elementary lemma.
\begin{lemma}\label{lm}
Let $\mathcal{X}$ be an inner product space. If $x$ and $y$ are linearly independent vectors of $\mathcal{X}$ and $t\in\mathbb{R}$,
then
\begin{equation*}
\lVert x+ty\rVert=\frac{\sqrt{\lVert x\rVert^{2}\lVert y\rVert^{2}-\langle x,y\rangle^{2}}}{\lVert y\rVert}\sum_{k=0}^{\infty}\binom{\frac{1}{2}}{k}\bigg(\frac{t\lVert y\rVert^{2}+\langle x,y\rangle}{\sqrt{\lVert x\rVert^{2}\lVert y\rVert^{2}-\langle x,y\rangle^{2}}}\bigg)^{2k},
\end{equation*}
whenever,
\begin{equation}\label{hg}
\frac{-\langle x,y\rangle-\sqrt{\lVert x\rVert^{2}\lVert y\rVert^{2}-\langle x,y\rangle^{2}}}{\lVert y\rVert^{2}}\leq t\leq\frac{-\langle x,y\rangle+\sqrt{\lVert x\rVert^{2}\lVert y\rVert^{2}-\langle x,y\rangle^{2}}}{\lVert y\rVert^{2}}.
\end{equation}
\end{lemma}
\begin{proof}
Employing the binomial series \cite{kh}, we get
\begin{align*}
\lVert x+ty\rVert&=(t^{2}\lVert y\rVert^{2}+2t\langle x,y\rangle+\lVert x\rVert^{2})^{\frac{1}{2}}\\
&=\bigg[\bigg(t\lVert y\rVert+\frac{\langle x,y\rangle}{\lVert y\rVert}\bigg)^{2}+\lVert x\rVert^{2}-\frac{\langle x,y\rangle^{2}}{\lVert y\rVert^{2}}\bigg]^{\frac{1}{2}}\\
&=\frac{\sqrt{\lVert x\rVert^{2}\lVert y\rVert^{2}-\langle x,y\rangle^{2}}}{\lVert y\rVert}\bigg[\bigg(\frac{t\lVert y\rVert^{2}+\langle x,y\rangle}{\sqrt{\lVert x\rVert^{2}\lVert y\rVert^{2}-\langle x,y\rangle^{2}}}\bigg)^{2}+1\bigg]^{\frac{1}{2}}\\
&=\frac{\sqrt{\lVert x\rVert^{2}\lVert y\rVert^{2}-\langle x,y\rangle^{2}}}{\lVert y\rVert}\sum_{k=0}^{\infty}\binom{\frac{1}{2}}{k}\bigg(\frac{t\lVert y\rVert^{2}+\langle x,y\rangle}{\sqrt{\lVert x\rVert^{2}\lVert y\rVert^{2}-\langle x,y\rangle^{2}}}\bigg)^{2k},
\end{align*}
whenever
\begin{equation*}
\bigg\lvert\frac{t\lVert y\rVert^{2}+\langle x,y\rangle}{\sqrt{\lVert x\rVert^{2}\lVert y\rVert^{2}-\langle x,y\rangle^{2}}}\bigg\rvert\leq 1,
\end{equation*}
which is equivalent to $(\ref{hg})$.
\end{proof}
\begin{theorem}
Let $\mathcal{X}$ be an inner product space and $p\in\mathbb{R}$. If $x$ and $y$ are linearly independent vectors of $\mathcal{X}$,
then
{\footnotesize\begin{equation}\label{kkk}
\alpha_p[x,y]=\frac{\sqrt{\lVert x\rVert^{2}\lVert y\rVert^{2}-\langle x,y\rangle^{2}}}{\lVert x\rVert^{1-p}\lVert y\rVert}\sum_{k=0}^{\infty}\binom{\frac{1}{2}}{k}\bigg(\frac{\lVert x\rVert^{1-p}\lVert y\rVert^{1+p}-\langle x,y\rangle}{\sqrt{\lVert x\rVert^{2}\lVert y\rVert^{2}-\langle x,y\rangle^{2}}}\bigg)^{2k},
\end{equation}}
whenever $\lVert y\rVert^{p}\leq\sqrt{2}\lVert x\rVert^{p}$ and
{\tiny\begin{equation}\label{okl}
\Big(-1\leq\Big)\frac{\lVert x\rVert^{-p}\lVert y\rVert^{p}-\sqrt{2-\lVert x\rVert^{-2p}\lVert y\rVert^{2p}}}{2}\leq\frac{\langle x,y\rangle}{\lVert x\rVert\lVert y\rVert}\leq\frac{\lVert x\rVert^{-p}\lVert y\rVert^{p}+\sqrt{2-\lVert x\rVert^{-2p}\lVert y\rVert^{2p}}}{2}\Big(\leq1\Big).
\end{equation}}
Similar expansion holds if we change the roles of $x$ and $y$ with each other.
\end{theorem}
\begin{proof}
We have
\small\begin{equation*}
\alpha_p[x,y]=\left\lVert\lVert x\rVert^{p-1}x-\lVert y\rVert^{p-1}y\right\rVert
=\lVert x\rVert^{p-1}\bigg\lVert x-\frac{\lVert y\rVert^{p-1}}{\lVert x\rVert^{p-1}}y\bigg\rVert.
\end{equation*}
Taking $t=-\frac{\lVert y\rVert^{p-1}}{\lVert x\rVert^{p-1}}$ in Lemma \ref{lm}, we reach
\small\begin{equation*}
\alpha_p[x,y]=\lVert x\rVert^{p-1}\frac{\sqrt{\lVert x\rVert^{2}\lVert y\rVert^{2}-\langle x,y\rangle^{2}}}{\lVert y\rVert}\sum_{k=0}^{\infty}\binom{\frac{1}{2}}{k}\bigg(\frac{\lVert x\rVert^{1-p}\lVert y\rVert^{1+p}-\langle x,y\rangle}{\sqrt{\lVert x\rVert^{2}\lVert y\rVert^{2}-\langle x,y\rangle^{2}}}\bigg)^{2k},
\end{equation*}
provided that,
\begin{equation*}
\frac{-\langle x,y\rangle-\sqrt{\lVert x\rVert^{2}\lVert y\rVert^{2}-\langle x,y\rangle^{2}}}{\lVert y\rVert^{2}}\leq -\frac{\lVert y\rVert^{p-1}}{\lVert x\rVert^{p-1}}\leq\frac{-\langle x,y\rangle+\sqrt{\lVert x\rVert^{2}\lVert y\rVert^{2}-\langle x,y\rangle^{2}}}{\lVert y\rVert^{2}}.
\end{equation*}
But, this condition is in turn equivalent to
\begin{equation*}
\bigg\lvert\langle x,y\rangle-\lVert x\rVert^{1-p}\lVert y\rVert^{1+p}\bigg\rvert\leq\sqrt{\lVert x\rVert^{2}\lVert y\rVert^{2}-\langle x,y\rangle^{2}},
\end{equation*}
or
\begin{equation*}
2\langle x,y\rangle^{2}-2\lVert x\rVert^{1-p}\lVert y\rVert^{1+p}\langle x,y\rangle+\lVert x\rVert^{2}\lVert y\rVert^{2}(\lVert x\rVert^{-2p}\lVert y\rVert^{2p}-1)\leq0,
\end{equation*}
which is equivalent to $\lVert x\rVert^{2}\lVert y\rVert^{2}(2-\lVert x\rVert^{-2p}\lVert y\rVert^{2p})\geq0$ and $(\ref{okl})$.
\end{proof}
The following corollary shows that $\alpha[x,y]$ is completely expressible by $\lVert x\rVert$, $\lVert y\rVert$ and the errors of the Cauchy-Schwarz inequality; $\lVert x\rVert\lVert y\rVert\pm\langle x,y\rangle\geq0$.
\begin{corollary}
Let $\mathcal{X}$ be an inner product space. If $x$ and $y$ are linearly independent vectors of $\mathcal{X}$,
then
\small\begin{equation}\label{hf}
\alpha[x,y]=\frac{\sqrt{\lVert x\rVert^{2}\lVert y\rVert^{2}-\langle x,y\rangle^{2}}}{\lVert x\rVert\lVert y\rVert}\sum_{k=0}^{\infty}\binom{\frac{1}{2}}{k}\bigg(\frac{\lVert x\rVert\lVert y\rVert-\langle x,y\rangle}{\lVert x\rVert\lVert y\rVert+\langle x,y\rangle}\bigg)^{k},
\end{equation}
whenever $\langle x,y\rangle\geq 0$, and
\small\begin{equation}\label{fh}
\alpha[x,y]=\sqrt{4-\frac{\lVert x\rVert^{2}\lVert y\rVert^{2}-\langle x,y\rangle^{2}}{\lVert x\rVert^{2}\lVert y\rVert^{2}}\bigg[\sum_{k=0}^{\infty}\binom{\frac{1}{2}}{k}\bigg(\frac{\lVert x\rVert\lVert y\rVert+\langle x,y\rangle}{\lVert x\rVert\lVert y\rVert-\langle x,y\rangle}\bigg)^{k}\bigg]^{2}},
\end{equation}
whenever $\langle x,y\rangle<0$.
\end{corollary}
\begin{proof}
The equality $(\ref{hf})$ follows from $(\ref{kkk})$ by taking $p=0$. If $\langle x,y\rangle<0$, then $\langle x,-y\rangle> 0$, and so $(\ref{fh})$ follows from $(\ref{hf})$ and $\alpha[x,y]=\sqrt{4-\alpha^{2}[x,-y]}$.
\end{proof}

\section{Comparison of $p$-angular and $q$-angular distances}

In this section, we compare two quantities $\alpha_p$ with $\alpha_q$ for arbitrary $p,q\in\mathbb{R}$. There are several papers related to comparison of $\alpha_p$ with $\alpha_1$; see, e.g., \cite{D2}-\cite{D9}. The advantage of taking $p$ and $q$ arbitrary is that, whenever we find an inequality involving $\alpha_p$ and $\alpha_q$, we can obtain its reverse by changing the roles of $p$ and $q$ with each other, which is as sharp as the first one.

\subsection{Generalizations of Maligranda's results}

The following theorem is a generalization of Maligranda's inequalities \cite{M}.
\begin{theorem}\label{po}
Let $p,q\in\mathbb{R}$, $q\neq0$ and $x,y\in \mathcal{X}\smallsetminus{\lbrace 0\rbrace}$.
\begin{itemize}
\item [\rm{(i)}]~ If $\frac{p}{q}\geq 1$, then
\begin{align}\label{i1}
\frac{p}{2p-q}&\max(\lVert x\rVert^{p-q},\lVert y\rVert^{p-q})\alpha_q[x,y]\nonumber\\
&\leq\alpha_p[x,y]\nonumber\\
&\leq \frac{p}{q}\max(\lVert x\rVert^{p-q},\lVert y\rVert^{p-q})\alpha_q[x,y].
\end{align}
\item [\rm{(ii)}]~ If $0\leq\frac{p}{q}\leq 1$, then
\begin{align}\label{i2}
\frac{p}{q}&\cdot\frac{\alpha_q[x,y]}{\max(\lVert x\rVert^{q-p},\lVert y\rVert^{q-p})}\nonumber\\
&\leq\alpha_p[x,y]\nonumber\\
&\leq\frac{2q-p}{q}\cdot\frac{\alpha_q[x,y]}{\max(\lVert x\rVert^{q-p},\lVert y\rVert^{q-p})}.
\end{align}
\item [\rm{(iii)}]~ If $\frac{p}{q}\leq 0$, then
\begin{align}\label{i3}
\frac{p}{2p-q}&\cdot\frac{\max(\lVert x\rVert^{p},\lVert y\rVert^{p})}{\max(\lVert x\rVert^{q},\lVert y\rVert^{q})}\alpha_q[x,y]\nonumber\\
&\leq\alpha_p[x,y]\nonumber\\
&\leq\frac{2q-p}{q}\cdot\frac{\max(\lVert x\rVert^{p},\lVert y\rVert^{p})}{\max(\lVert x\rVert^{q},\lVert y\rVert^{q})}\alpha_q[x,y].
\end{align}
\end{itemize}
\end{theorem}
\begin{proof}
The left inequalities are obtained from right ones by interchanging the roles of $p$ and $q$. So, it is sufficient to prove only the right inequalities.
Without loss of generality we may assume that $\lVert x\rVert^{q}\leq\lVert y\rVert^{q}$. By the triangle inequality, we have
\small\begin{align*}
\alpha_p[x,y]&\leq
\left\lVert\lVert x\rVert^{p-q}\lVert x\rVert^{q-1}x-\lVert y\rVert^{p-q}\lVert x\rVert^{q-1}x\right\rVert\\
&+\left\lVert\lVert y\rVert^{p-q}\lVert x\rVert^{q-1}x-\lVert y\rVert^{p-q}\lVert y\rVert^{q-1}y\right\rVert\\
&=\lVert x\rVert^{q}\left\lvert\lVert x\rVert^{p-q}-\lVert y\rVert^{p-q}\right\rvert+\lVert y\rVert^{p-q}\alpha_q[x,y].
\end{align*}
$\quad(i)$ Let $\frac{p}{q}\geq1$. Since $\frac{p-q}{q}\geq0$, we have $\lVert x\rVert^{p-q}\leq\lVert y\rVert^{p-q}$, and so
\begin{equation*}
\alpha_p[x,y]\leq\lVert x\rVert^{q}(\lVert y\rVert^{p-q}-\lVert x\rVert^{p-q})+\lVert y\rVert^{p-q}\alpha_q[x,y].
\end{equation*}
If $\frac{p}{q}\geq2$, then since $\frac{p-2q}{q}\geq0$, we get
\begin{equation*}
\lVert y\rVert^{p-q}-\lVert x\rVert^{p-q}=\frac{p-q}{q}\int_{\lVert x\rVert^{q}}^{\lVert y\rVert^{q}}t^{\frac{p-2q}{q}}dt\leq\frac{p-q}{q}\lVert y\rVert^{p-2q}(\lVert y\rVert^{q}-\lVert x\rVert^{q}),
\end{equation*}
which leads to
\begin{align*}
\lVert x\rVert^{q}(\lVert y\rVert^{p-q}-\lVert x\rVert^{p-q})&\leq\frac{p-q}{q}\lVert x\rVert^{q}\lVert y\rVert^{p-2q}(\lVert y\rVert^{q}-\lVert x\rVert^{q})\\
&\leq\frac{p-q}{q}\lVert y\rVert^{p-q}\alpha_q[x,y].
\end{align*}
Whence
\begin{equation*}
\alpha_p[x,y]\leq\frac{p}{q}\lVert y\rVert^{p-q}\alpha_q[x,y].
\end{equation*}
If $1\leq\frac{p}{q}\leq 2$, it follows from $\frac{p-2q}{q}\leq0$ that
\begin{equation*}
\lVert y\rVert^{p-q}-\lVert x\rVert^{p-q}=\frac{p-q}{q}\int_{\lVert x\rVert^{q}}^{\lVert y\rVert^{q}}t^{\frac{p-2q}{q}}dt\leq\frac{p-q}{q}\lVert x\rVert^{p-2q}(\lVert y\rVert^{q}-\lVert x\rVert^{q}),
\end{equation*}
which gives
\begin{equation*}
\lVert x\rVert^{q}(\lVert y\rVert^{p-q}-\lVert x\rVert^{p-q})\leq\frac{p-q}{q}\lVert x\rVert^{p-q}(\lVert y\rVert^{q}-\lVert x\rVert^{q})
\leq\frac{p-q}{q}\lVert y\rVert^{p-q}\alpha_q[x,y],
\end{equation*}
and again
\begin{equation*}
\alpha_p[x,y]\leq\frac{p}{q}\lVert y\rVert^{p-q}\alpha_q[x,y].
\end{equation*}
$\quad(ii)$ Let $0\leq\frac{p}{q}\leq 1$. The inequality $\frac{q-p}{q}\geq0$ yields that $\lVert x\rVert^{q-p}\leq\lVert y\rVert^{q-p}$, and so
\begin{equation*}
\lVert x\rVert^{q}\left\lvert\lVert x\rVert^{p-q}-\lVert y\rVert^{p-q}\right\rvert=\lVert x\rVert^{q}\frac{\lVert y\rVert^{q-p}-\lVert x\rVert^{q-p}}{\lVert x\rVert^{q-p}\lVert y\rVert^{q-p}}.
\end{equation*}
It follows from
\begin{equation*}
\lVert y\rVert^{q-p}-\lVert x\rVert^{q-p}=\frac{q-p}{q}\int_{\lVert x\rVert^{q}}^{\lVert y\rVert^{q}}t^{-\frac{p}{q}}dt\leq\frac{q-p}{q}\lVert x\rVert^{-p}(\lVert y\rVert^{q}-\lVert x\rVert^{q}),
\end{equation*}
that
\begin{equation*}
\lVert x\rVert^{q}\frac{\lVert y\rVert^{q-p}-\lVert x\rVert^{q-p}}{\lVert x\rVert^{q-p}\lVert y\rVert^{q-p}}\leq\frac{q-p}{q}\cdot\frac{\lVert y\rVert^{q}-\lVert x\rVert^{q}}{\lVert y\rVert^{q-p}}.
\end{equation*}
Hence
\begin{equation*}
\alpha_p[x,y]\leq\frac{q-p}{q}\cdot\frac{\lVert y\rVert^{q}-\lVert x\rVert^{q}}{\lVert y\rVert^{q-p}}+\frac{\alpha_q[x,y]}{\lVert y\rVert^{q-p}}
\leq\Big(2-\frac{p}{q}\Big)\frac{\alpha_q[x,y]}{\lVert y\rVert^{q-p}}.
\end{equation*}
$\quad(iii)$ The same reasoning as in the proof of $(ii)$ yields $(iii)$.
\end{proof}
Now, taking $q=1$ in Theorem \ref{po}, we obtain the following corollary in which the right inequalities are due to Maligranda \cite{M} and left ones are new suitable reverses to them.
\begin{corollary}
Let $x,y\in \mathcal{X}\smallsetminus{\lbrace 0\rbrace}$.
\begin{itemize}
\item [\rm{(i)}]~ If $p\geq 1$, then
\small\begin{equation*}
\frac{p}{2p-1}\max(\lVert x\rVert,\lVert y\rVert)^{p-1}\lVert x-y\rVert\leq\alpha_p[x,y]\leq p\max(\lVert x\rVert,\lVert y\rVert)^{p-1}\lVert x-y\rVert.
\end{equation*}
\item [\rm{(ii)}]~ If $0\leq p\leq 1$, then
\small\begin{equation*}
\frac{p\lVert x-y\rVert}{\max(\lVert x\rVert,\lVert y\rVert)^{1-p}}\leq\alpha_p[x,y]\leq\frac{(2-p)\lVert x-y\rVert}{\max(\lVert x\rVert,\lVert y\rVert)^{1-p}}.
\end{equation*}
\item [\rm{(iii)}]~ If $p\leq 0$, then
\small\begin{equation*}
\frac{p}{2p-1}\cdot\frac{\max(\lVert x\rVert^{p},\lVert y\rVert^{p})}{\max(\lVert x\rVert,\lVert y\rVert)}\lVert x-y\rVert\leq\alpha_p[x,y]\leq(2-p)\frac{\max(\lVert x\rVert^{p},\lVert y\rVert^{p})}{\max(\lVert x\rVert,\lVert y\rVert)}\lVert x-y\rVert.
\end{equation*}
\end{itemize}
\end{corollary}
\begin{corollary}
Let $p\neq 2$ and $x,y\in \mathcal{X}\smallsetminus{\lbrace 0\rbrace}$.
\begin{itemize}
\item [\rm{(i)}]~ If $\frac{p}{2-p}\geq 1$, then
\begin{align*}
\frac{p}{3p-2}&\max(\lVert x\rVert^{p-1}\lVert y\rVert^{1-p},\lVert y\rVert^{p-1}\lVert x\rVert^{1-p})\beta_p[x,y]\\
&\leq\alpha_p[x,y]\\
&\leq\frac{p}{2-p}\max(\lVert x\rVert^{p-1}\lVert y\rVert^{1-p},\lVert y\rVert^{p-1}\lVert x\rVert^{1-p})\beta_p[x,y].
\end{align*}
\item [\rm{(ii)}]~ If $0\leq\frac{p}{2-p}\leq 1$, then
\begin{align*}
\frac{p}{(2-p)}&\cdot\frac{\beta_p[x,y]}{\max(\lVert x\rVert^{p-1}\lVert y\rVert^{1-p},\lVert y\rVert^{p-1}\lVert x\rVert^{1-p})}\\
&\leq\alpha_p[x,y]\\
&\leq\frac{4-3p}{2-p}\cdot\frac{\beta_p[x,y]}{\max(\lVert x\rVert^{p-1}\lVert y\rVert^{1-p},\lVert y\rVert^{p-1}\lVert x\rVert^{1-p})}.
\end{align*}
\item [\rm{(iii)}]~If $\frac{p}{2-p}\leq 0$, then
\begin{align*}
\frac{p}{3p-2}&\cdot\frac{\max(\lVert x\rVert^{p},\lVert y\rVert^{p})}{\max(\lVert x\rVert\lVert y\rVert^{p-1},\lVert y\rVert\lVert x\rVert^{p-1})}\beta_p[x,y]\\
&\leq\alpha_p[x,y]\\
&\leq\frac{4-3p}{2-p}\cdot\frac{\max(\lVert x\rVert^{p},\lVert y\rVert^{p})}{\max(\lVert x\rVert\lVert y\rVert^{p-1},\lVert y\rVert\lVert x\rVert^{p-1})}\beta_p[x,y].
\end{align*}
\end{itemize}
In particular, for $p=0$ and $q=1$, it follows from $(ii)$ that
\begin{equation*}
\alpha[x,y]\leq2\min\left\{\frac{\lVert x\rVert}{\lVert y\rVert},\frac{\lVert y\rVert}{\lVert x\rVert}\right\}\beta[x,y].
\end{equation*}
\end{corollary}
\begin{proof}
Take $q=2-p$ in Theorem \ref{po} and consider $(\ref{pp})$.
\end{proof}
\begin{remark}
In $(\ref{i1}),(\ref{i2})$ and $(\ref{i3})$, the constants $2-\frac{p}{q}$ and $\frac{p}{q}$ in the right  inequalities are best possible. In fact, consider $\mathcal{X}=\mathbb{R}^{2}$ with the norm of $x=(x_1,x_2)$ given by $\lVert x\rVert=\lvert x_1\rvert+\lvert x_2\rvert$. Take $x=(1+\epsilon)^{\frac{1-q}{q}}(1,\epsilon)$ and $y=(1,0)$, where $\epsilon>0$ is small. If $\frac{p}{q}\geq1$, then
\begin{equation*}
\frac{\alpha_p[x,y]}{\alpha_q[x,y]\max(\lVert x\rVert^{p-q},\lVert y\rVert^{p-q})}=\frac{(1+\epsilon)^{\frac{p}{q}-1}-1}{\epsilon}\cdot\frac{1}{(1+\epsilon)^{\frac{p}{q}-1}}+1\rightarrow\frac{p}{q}
\end{equation*}
as $\epsilon\rightarrow 0^{+}$. In the case $0\leq\frac{p}{q}\leq1$, we obtain
\begin{equation*}
\frac{\alpha_p[x,y]}{\alpha_q[x,y]}\max(\lVert x\rVert^{q-p},\lVert y\rVert^{q-p})=\frac{(1+\epsilon)^{1-\frac{p}{q}}-1}{\epsilon}+1\rightarrow 2-\frac{p}{q}
\end{equation*}
as $\epsilon\rightarrow 0^{+}$. 
\par In the case when $\frac{p}{q}\leq0$, the best possibility of the constant $2-\frac{p}{q}$   in the right inequality of $(\ref{i3})$ is similarly verified. 
The best possibility of constants $\frac{p}{2p-q}$ and $\frac{p}{q}$ in the left inequalities of $(\ref{i1}),(\ref{i2})$ and $(\ref{i3})$ are obtained from the best possibility of constants $2-\frac{p}{q}$ and $\frac{p}{q}$ in the right hand sides of these inequalities by changing the roles of $p$ and $q$.
\end{remark}
\begin{remark}
Let $p,q\in\mathbb{R}$ and $q\neq0$. It is easily seen that in the case where $\frac{p}{q}\geq1$ (resp. $0\leq\frac{p}{q}\leq1$), the right (resp. left) hand side of inequality $(\ref{i1})$ (resp. $(\ref{i2})$) is as the same as the right (resp. left) hand side of inequality $(\ref{1})$, but the left (resp. right) hand side of inequality $(\ref{i1})$ (resp. $(\ref{i2})$) gives better estimate than the left (resp. right) hand side of inequality $(\ref{1})$. In the case when $\frac{p}{q}\leq 0$, using the fact that
\begin{equation*}
\min\Big(\frac{a}{c},\frac{b}{d}\Big)\leq\frac{\max(a,b)}{\max(c,d)}\leq\max\Big(\frac{a}{c},\frac{b}{d}\Big)\qquad(a,b,c,d>0),
\end{equation*}
both sides of $(\ref{i3})$ are better estimates than both corresponding sides of $(\ref{1})$.
\end{remark}

\subsection{Generalization of Dragomir's results}

The following theorem yields the result of Dragomir in \cite{D2}, if we take $q=1$.
\begin{theorem}
Let $x,y\in \mathcal{X}\smallsetminus{\lbrace 0\rbrace}$, $p,q\in\mathbb{R}$ and $q\neq0$.
\begin{itemize}
\item [\rm{(i)}]~ If $\frac{p}{q}\geq 1$, then
\begin{equation}\label{nn}
\alpha_p[x,y]\leq\frac{p}{q}~\alpha_q[x,y]\int_{0}^{1}\left\lVert(1-t)\lVert x\rVert^{q-1}x+t\lVert y\rVert^{q-1}y\right\rVert^{\frac{p}{q}-1}dt.
\end{equation}
\item [\rm{(ii)}]~ If $\frac{p}{q}<1$ and $x,y$ are linearly independent, then
\small{\begin{equation}\label{mm}
\alpha_p[x,y]\leq\frac{2q-p}{q}~\alpha_q[x,y]\int_{0}^{1}\left\lVert(1-t)\lVert x\rVert^{q-1}x+t\lVert y\rVert^{q-1}y\right\rVert^{\frac{p}{q}-1}dt.
\end{equation}}
\end{itemize}
\end{theorem}
\begin{proof}
We suppose that $x,y$ are linearly independent and prove $(\ref{nn})$ and $(\ref{mm})$ by one strike. As one can observe, this proof works also in the case when $\frac{p}{q}\geq 1$ and $x,y$ are linearly dependent. The function
$f:[0,1]\rightarrow[0,\infty)$ given by
\begin{equation*}
f(t)=\left\lVert(1-t)\lVert x\rVert^{q-1}x+ t\lVert y\rVert^{q-1}y\right\rVert^{\frac{p}{q}-1},
\end{equation*}
and the vector-valued function $h:[0,1]\rightarrow \mathcal{X}$ given by
\begin{equation*}
h(t)=\big[(1-t)\lVert x\rVert^{q-1}x+t\lVert y\rVert^{q-1}y\big],
\end{equation*}
are both absolutely continuous on
$[0,1]$.  Therefore, the function $g:[0,1]\rightarrow \mathcal{X}$ given by $g(t)=f(t)h(t)$ is absolutely continuous. The function $k(t):=\left\lVert(1-t)\lVert x\rVert^{q-1}x+ t\lVert y\rVert^{q-1}y\right\rVert$ is convex, and so except than at most a countable number of points, $k'(t)$ exists. It is easily verified that $\lvert k'(t)\rvert\leq\alpha_q[x,y]$ in each differentiability point $t$. We have
\begin{align*}
g'(t)&=f'(t)h(t)+f(t)h'(t)\\
&=\Big(\frac{p}{q}-1\Big)k(t)^{\frac{p}{q}-2}k'(t)h(t)
+f(t)\big[\lVert y\rVert^{q-1}y-\lVert x\rVert^{q-1}x\big]
\end{align*}
for almost all $t\in[0,1]$. Thus,
\small\begin{equation*}
\lVert g'(t)\rVert\leq\bigg(\bigg\lvert\frac{p}{q}-1\bigg\rvert+1\bigg)\left\lVert(1-t)\lVert x\rVert^{q-1}x+ t\lVert y\rVert^{q-1}y\right\rVert^{\frac{p}{q}-1}\alpha_q[x,y]
\end{equation*}
for almost all $t\in[0,1]$. Utilizing the norm inequality for the vector-valued integral, we get
\Small\begin{align*}
\alpha_p[x,y]&=\left\lVert\lVert y\rVert^{p-1}y-\lVert x\rVert^{p-1}x\right\rVert
=\lVert g(1)-g(0)\rVert
=\bigg\lVert\int_{0}^{1}g'(t)dt\bigg\rVert
\leq\int_{0}^{1}\lVert g'(t)\rVert dt\\
&\leq\bigg(\bigg\lvert\frac{p}{q}-1\bigg\rvert+1\bigg)\alpha_q[x,y]\int_{0}^{1}\left\lVert(1-t)\lVert x\rVert^{q-1}x+ t\lVert y\rVert^{q-1}y\right\rVert^{\frac{p}{q}-1}dt,
\end{align*}
and so, the proofs of $(\ref{nn})$ and $(\ref{mm})$ are complete.
\end{proof}
\begin{corollary}
Let $x,y\in \mathcal{X}$ be linearly independent and $p,q\in\mathbb{R}\smallsetminus{\lbrace 0\rbrace}$.
\begin{itemize}
\item [\rm{(i)}]~ If $0<\frac{p}{q}\leq1$, then
\begin{equation*}
\alpha_p[x,y]\geq\frac{p}{q}\alpha_q[x,y]\bigg(\int_{0}^{1}\left\lVert(1-t)\lVert x\rVert^{p-1}x+t\lVert y\rVert^{p-1}y\right\rVert^{\frac{q}{p}-1}dt\bigg)^{-1}.
\end{equation*}
\item [\rm{(ii)}]~ If $\frac{p}{q}\geq 1$ or $\frac{p}{q}<0$, then
\begin{equation*}
\alpha_p[x,y]\geq\frac{p}{2p-q}\alpha_q[x,y]\bigg(\int_{0}^{1}\left\lVert(1-t)\lVert x\rVert^{p-1}x+t\lVert y\rVert^{p-1}y\right\rVert^{\frac{q}{p}-1}dt\bigg)^{-1}.
\end{equation*}
\end{itemize}
\end{corollary}
\begin{remark}
$(i)$ If $\frac{p}{q}\geq 1$, then, by the triangle inequality, we have
\begin{equation*}
\left\lVert(1-t)\lVert x\rVert^{q-1}x+ t\lVert y\rVert^{q-1}y\right\rVert^{\frac{p}{q}-1}\leq[(1-t)\lVert x\rVert^{q}+t\lVert y\rVert^{q}]^{{\frac{p}{q}-1}}
\end{equation*}
for any $t\in[0,1]$. Integrating both sides on $[0,1]$, we get
\begin{equation*}
\int_{0}^{1}\left\lVert(1-t)\lVert x\rVert^{q-1}x+ t\lVert y\rVert^{q-1}y\right\rVert^{\frac{p}{q}-1}dt\leq\frac{q}{p}\bigg(\frac{\lVert y\rVert^{p}-\lVert x\rVert^{p}}{\lVert y\rVert^{q}-\lVert x\rVert^{q}}\bigg)
\end{equation*}
if $\lVert x\rVert\neq\lVert y\rVert$, and by $(\ref{nn})$ we obtain the chain of inequalities
\begin{align}\label{vv}
\alpha_p[x,y]&\leq\frac{p}{q}~\alpha_q[x,y]\int_{0}^{1}\left\lVert(1-t)\lVert x\rVert^{q-1}x+t\lVert y\rVert^{q-1}y\right\rVert^{\frac{p}{q}-1}dt\nonumber\\
&\leq\frac{\lVert y\rVert^{p}-\lVert x\rVert^{p}}{\lVert y\rVert^{q}-\lVert x\rVert^{q}}\alpha_q[x,y],
\end{align}
which provides a generalization and refinement of Hile's inequality $(\ref{nbn})$.\par
$(ii)$ If $\frac{p}{q}\geq2$, then the function $f:[0,1]\rightarrow[0,\infty)$ given by $f(t)=[(1-t)\lVert x\rVert^{q}+t\lVert y\rVert^{q}]^{{\frac{p}{q}-1}}$ is convex. Employing the Hermite-Hadamard inequality for the convex function $f$ (see \cite{MSM} and references therein) we obtain
\begin{align*}
\frac{q}{p}\bigg(\frac{\lVert y\rVert^{p}-\lVert x\rVert^{p}}{\lVert y\rVert^{q}-\lVert x\rVert^{q}}\bigg)
&=\int_{0}^{1}[(1-t)\lVert x\rVert^{q}+t\lVert y\rVert^{q}]^{{\frac{p}{q}-1}}dt\\
&\leq\frac{\lVert x\rVert^{p-q}+\lVert y\rVert^{p-q}}{2}\leq\max\{\lVert x\rVert^{p-q},\lVert y\rVert^{p-q}\},
\end{align*}
which by $(\ref{nn})$, implies the following sequence of inequalities
\begin{align}\label{cch}
\alpha_p[x,y]&\leq\frac{p}{q}~\alpha_q[x,y]\int_{0}^{1}\left\lVert(1-t)\lVert x\rVert^{q-1}x+t\lVert y\rVert^{q-1}y\right\rVert^{\frac{p}{q}-1}dt\nonumber\\
&\leq\frac{p}{q}~\alpha_q[x,y]\int_{0}^{1}[(1-t)\lVert x\rVert^{q}+t\lVert y\rVert^{q}]^{{\frac{p}{q}-1}}dt\nonumber\\
&=\frac{\lVert y\rVert^{p}-\lVert x\rVert^{p}}{\lVert y\rVert^{q}-\lVert x\rVert^{q}}\alpha_q[x,y]\nonumber\\
&\leq\frac{p}{q}\alpha_q[x,y]\frac{\lVert x\rVert^{p-q}+\lVert y\rVert^{p-q}}{2}\leq\frac{p}{q}\alpha_q[x,y]\max\{\lVert x\rVert^{p-q},\lVert y\rVert^{p-q}\}
\end{align}
for $\lVert x\rVert\neq\lVert y\rVert$.
\par In particular, inequality $(\ref{cch})$ shows that in the case $\frac{p}{q}\geq2$, inequality
$(\ref{vv})$ is better than inequality $(\ref{i1})$.
\end{remark}
\begin{remark}
Let $\mathcal{X}$ be an inner product space. It is known \cite{D5} that for any $a,b\in \mathcal{X}$, $b\neq0$, it holds that
\begin{equation*}
\min_{t\in\mathbb{R}}\lVert a+t b\rVert=\frac{\sqrt{\lVert a\rVert^{2}\lVert b\rVert^{2}-\lvert\langle a,b\rangle\rvert^{2}}}{\lVert b\rVert}.
\end{equation*}
Hence, if $x$ and $y$ are linearly independent vectors of $\mathcal{X}$, then by taking $a=x$ and $b=y-x$, we obtain
\begin{equation*}
\lVert(1-t)x+ty\rVert=\lVert x+t(y-x)\rVert\geq\frac{\sqrt{\lVert x\rVert^{2}\lVert y\rVert^{2}-\langle x,y\rangle^{2}}}{\lVert x-y\rVert}\quad(t\in\mathbb{R}).
\end{equation*}
This implies that
\begin{equation*}
\int_{0}^{1}\lVert(1-t)x+ty\rVert^{-1}dt\leq\frac{\lVert x-y\rVert}{\sqrt{\lVert x\rVert^{2}\lVert y\rVert^{2}-\langle x,y\rangle^{2}}}.
\end{equation*}
Taking $p=0$ and $q=1$ in $(\ref{mm})$, we get
\begin{equation*}
\alpha[x,y]\leq 2\lVert x-y\rVert\int_{0}^{1}\lVert(1-t)x+ty\rVert^{-1}dt\leq\frac{2\lVert x-y\rVert^{2}}{\sqrt{\lVert x\rVert^{2}\lVert y\rVert^{2}-\langle x,y\rangle^{2}}}.
\end{equation*}
This implies an upper estimation for the error of the Cauchy-Schwarz inequality as follows
\begin{equation*}
\sqrt{\lVert x\rVert^{2}\lVert y\rVert^{2}-\langle x,y\rangle^{2}}\leq\frac{2\lVert x\rVert\lVert y\rVert\lVert x-y\rVert^{2}}{\left\lVert\lVert y\rVert x-\lVert x\rVert y\right\rVert}\qquad(\lVert y\rVert x\neq\lVert x\rVert y).
\end{equation*}
\end{remark}

\section{Characterizations of inner product spaces}

In this section, corresponding to Propositions $\ref{667}$ and $\ref{867}$, we give two characterizations of inner product spaces regarding to the $p$-angular and the skew $p$-angular distances.
\par The following characterization extends a result of Dehghan \cite{H} from $p=0$ to an arbitrary real number $p\neq1$.
\begin{theorem}\label{imp}
Let $p>1$$~($$p<1$ resp.$)$ is a real number. Then a normed space $\mathcal{X}$ is an inner product space, if and only if for any $x,y\in \mathcal{X}\smallsetminus{\lbrace 0\rbrace}$,
\begin{equation}\label{10}
\alpha_{p}[x,y]\geq\beta_{p}[x,y]\quad(\alpha_{p}[x,y]\leq\beta_{p}[x,y]~\text{resp.}).
\end{equation}
\end{theorem}
\begin{proof}
If $\mathcal{X}$ is an inner product space, the conclusion follows from Proposition $\ref{667}$. Now, let $\mathcal{X}$ be a normed space satisfying the condition (\ref{10}). Since for arbitrary non-zero elements $x$ and $y$ of $\mathcal{X}$, the inequality $\alpha_{p}[x,y]\leq\beta_{p}[x,y]$ is equivalent to $\beta_{2-p}[x,y]\leq\alpha_{2-p}[x,y]$, it is sufficient to consider the case when $p>1$.
\par Let $x,y\in \mathcal{X}, \Vert x\Vert=\Vert y\Vert$ and $\gamma\neq 0$. From Theorem $\textbf{A}$ it is enough to prove that $\Vert\gamma x+\gamma^{-1}y\Vert\geq\Vert x+y\Vert$. Clearly, we can assume that $\lVert x\rVert=\lVert y\rVert=1$ and $\gamma >0$.
Applying inequality (\ref{10}) to $\gamma^{\frac{1}{p}} x$ and $-\gamma^{-\frac{1}{p}} y$ for $x$ and $y$ respectively, we obtain
\begin{equation}\label{234}
\Vert\gamma x+\gamma^{-1}y\Vert\geq\big\Vert\gamma^{\frac{2-p}{p}}x+\gamma^{-\frac{2-p}{p}}y\big\Vert.
\end{equation}
Now using the mathematical induction, we get
\begin{equation*}
\Vert\gamma x+\gamma^{-1}y\Vert\geq\bigg\Vert\gamma^{\big(\frac{2-p}{p}\big)^{n}}x+\gamma^{-\big(\frac{2-p}{p}\big)^{n}}y\bigg\Vert\quad(n=1,2,\ldots).
\end{equation*}
Since $p>1$, we have $\big\vert\frac{2-p}{p}\big\vert<1$, and so
\begin{equation*}\label{6}
\Vert\gamma x+\gamma^{-1}y\Vert\geq\lim_{n
\rightarrow\infty}\bigg\Vert\gamma^{\big(\frac{2-p}{p}\big)^{n}}x+\gamma^{-\big(\frac{2-p}{p}\big)^{n}}y\bigg\Vert=\Vert x+y\Vert.
\end{equation*}
This completes the proof.
\end{proof}
\begin{remark}
If $\mathcal{X}$ is not an inner product space, then for each $p\neq1$ there exist $x_i,y_i\in \mathcal{X}\smallsetminus{\lbrace 0\rbrace}~(i=1,2)$, such that $\alpha_p[x_1,y_1]<\beta_p[x_1,y_1]$ and $\alpha_p[x_2,y_2]>\beta_p[x_2,y_2]$. In fact if $p>1$, then by Theorem \ref{imp} there exist $x_1,y_1\in \mathcal{X}\smallsetminus{\lbrace 0\rbrace}$ such that $\alpha_p[x_1,y_1]<\beta_p[x_1,y_1]$. On the other hand, due to an arbitrary one dimensional subspace $M=\{\lambda e:
\lambda\in\mathbb{R}\}$ of $\mathcal{X}$ with $\lVert e\rVert=1$ is an inner product space via $\langle \lambda e, \mu e\rangle:=\lambda\mu$, for any $x_2,y_2\in M\smallsetminus{\lbrace 0\rbrace}$ with $\lVert x_2\rVert\neq\lVert y_2\rVert$, we have $\alpha_p[x_2,y_2]>\beta_p[x_2,y_2]$. A similar argument carry out when $p<1$.
\end{remark}
Now we give the second characterization of inner product spaces related to Proposition $\ref{867}$.
\begin{theorem}
Let $p\neq1$. Then a normed space $\mathcal{X}$ is an inner product space if and only if for any $x, y\in \mathcal{X}\smallsetminus{\lbrace 0\rbrace}$,
\small\begin{equation}\label{22}
\alpha_p[x, y]=\sqrt{(\lVert x\rVert^{p+1}-\lVert y\rVert^{p+1})(\lVert x\rVert^{p-1}-\lVert y\rVert^{p-1})+\lVert x\rVert^{p-1}\lVert y\rVert^{p-1}\lVert x-y\rVert^{2}}.
\end{equation}
\end{theorem}
\begin{proof}
If $\mathcal{X}$ is an inner product space, then identity $(\ref{22})$ follows from Proposition $\ref{867}$. Now, let $\mathcal{X}$ be a normed space satisfying condition $(\ref{22})$. We prove that $\mathcal{X}$ is an inner product space by considering the following three cases for $p$.\par
$\textbf{Case 1.}$ Assume that $p\neq 0, -1$. Let $x,y\in \mathcal{X}, \Vert x\Vert=\Vert y\Vert$ and $\lambda\neq 0$. From Theorem $\textbf{A}$ it is enough to prove that $\Vert x+y\Vert\leq\Vert\lambda x+\lambda^{-1}y\Vert$. We may
assume that $\lVert x\rVert=\lVert y\rVert=1$ and $\lambda>0$.
Applying identity $(\ref{22})$ to $\lambda^{\frac{1}{p}}x$ and $-\lambda^{-\frac{1}{p}}y$ for $x$ and $y$ respectively, we obtain
\begin{align*}
\lVert\lambda x+\lambda^{-1}y\rVert&=\sqrt{(\lambda^{\frac{p+1}{p}}-\lambda^{-\frac{p+1}{p}})(\lambda^{\frac{p-1}{p}}-\lambda^{\frac{1-p}{p}})+\lVert\lambda^{\frac{1}{p}}x+\lambda^{-\frac{1}{p}}y\rVert^{2}}\\
&=\sqrt{\frac{\lambda^{\frac{2p+2}{p}}-1}{\lambda^{\frac{p+1}{p}}}\cdot\frac{\lambda^{\frac{2p-2}{p}}-1}{\lambda^{\frac{p-1}{p}}}+\lVert\lambda^{\frac{1}{p}}x+\lambda^{-\frac{1}{p}}y\rVert^{2}}.
\end{align*}
If $p>1$, then $2p+2>2p-2>0$, and if $p<-1$, then $2p-2<2p+2<0$. For $\lvert p\rvert>1$, we therefore have $(\lambda^{\frac{2p+2}{p}}-1)(\lambda^{\frac{2p-2}{p}}-1)\geq0$. Hence $\lVert\lambda x+\lambda^{-1}y\rVert\geq\lVert\lambda^{\frac{1}{p}}x+\lambda^{-\frac{1}{p}}y\rVert$. It yields that
\begin{equation}\label{hh}
\lVert\lambda x+\lambda^{-1}y\rVert\geq\lVert\lambda^{\frac{1}{p^{n}}}x+\lambda^{-\frac{1}{p^{n}}}y\rVert\qquad(n=1,2,\ldots).
\end{equation}
Thus,
\begin{equation*}
\lVert\lambda x+\lambda^{-1}y\rVert\geq\lim_{n\rightarrow\infty}\lVert\lambda^{\frac{1}{p^{n}}}x+\lambda^{-\frac{1}{p^{n}}}y\rVert=\lVert x+y\rVert.
\end{equation*}
Now if $\lvert p\rvert<1$, then $\lvert\frac{1}{p}\rvert>1$, and so by substituting $p$ by $\frac{1}{p}$ in $(\ref{hh})$ we get
\begin{equation*}
\lVert\lambda x+\lambda^{-1}y\rVert\geq\lVert\lambda^{p^{n}}x+\lambda^{-p^{n}}y\rVert\qquad(n=1,2,\ldots).
\end{equation*}
Hence,
\begin{equation*}
\lVert\lambda x+\lambda^{-1}y\rVert\geq\lim_{n\rightarrow\infty}\lVert\lambda^{p^{n}}x+\lambda^{-p^{n}}y\rVert=\lVert x+y\rVert,
\end{equation*}
and so, $\mathcal{X}$ is an inner product space.\par
$\textbf{Case 2.}$ Suppose that $p=0$.
Let $x,y\in \mathcal{X}$, $\lVert x\rVert=\lVert y\rVert=1$ and $\lambda>0$.
Replacing $x$ and $y$ by $\lambda x$ and $-\lambda^{-1} y$ respectively, in identity $(\ref{22})$, we get
\begin{equation*}
\lVert x+y\rVert^{2}=\lVert\lambda x+\lambda^{-1}y\rVert^{2}-(\lambda-\frac{1}{\lambda})^{2}
\leq\lVert\lambda x+\lambda^{-1}y\rVert^{2}.
\end{equation*}
It follows from Theorem $\textbf{A}$ that $\mathcal{X}$ is an inner product space.\par
$\textbf{Case 3.}$
Let $p=-1$. Assume $x, y\in \mathcal{X}$ such that $\lVert x\rVert=\lVert y\rVert$ and $\lambda>0$. Applying identity $(\ref{22})$ to $\lambda x$ and $-\lambda^{-1}y$ instead of $x$ and $y$ respectively, we obtain $\lVert\lambda x+\lambda^{-1}y\rVert=\lVert\lambda^{-1}x+\lambda y\rVert$. Therefore, Theorem $\textbf{B}$ ensures that $\mathcal{X}$ is an inner product space.
\end{proof}
\begin{remark}\label{fg}
It seems that the characterization of inner product spaces in Theorem \ref{imp} can be extended in a more general case. For example, the following inequality
\small\begin{equation}\label{ls}
\bigg\lVert\frac{x}{1+\lVert x\rVert}-\frac{y}{1+\lVert y\rVert}\bigg\rVert\leq\bigg\lVert\frac{x}{1+\lVert y\rVert}-\frac{y}{1+\lVert x\rVert}\bigg\rVert\quad(x,y\in \mathcal{X}),
\end{equation}
is also a characterization of inner product spaces. In fact, $(\ref{ls})$ holds in any inner product spaces and conversely, if $(\ref{ls})$ holds in a normed linear space $\mathcal{X}$, then substituting $x$ and $y$ by $nx$ and $ny~(n=1,2,\ldots)$ respectively, we obtain
\begin{equation*}
\small\bigg\lVert\frac{x}{\frac{1}{n}+\lVert x\rVert}-\frac{y}{\frac{1}{n}+\lVert y\rVert}\bigg\rVert\leq\bigg\lVert\frac{x}{\frac{1}{n}+\lVert y\rVert}-\frac{y}{\frac{1}{n}+\lVert x\rVert}\bigg\rVert.
\end{equation*}
Now letting $n\to \infty$, we get $\alpha[x,y]\leq\beta[x,y]~(x,y\in \mathcal{X}\smallsetminus{\lbrace 0\rbrace})$, and so $\mathcal{X}$ is an inner product space.
\end{remark}
\bigskip
\subsection*{Acknowledgement}
The authors would like to sincerely thank the anonymous referee for carefully reading the article. The corresponding author, M. S. Moslehian, would like to thank The Tusi Mathematical Research Group (TMRG).
\bigskip
\bibliographystyle{amsplain}

\begin{thebibliography}{99}
\bibitem{Al} A. M. Al-Rashed, \textit{Norm inequalities and characterizations of inner product spaces}, J. Math. Anal. Appl. \textbf{176} (1993) 587--593.

\bibitem{A} D. Amir, \textit{Characterizations of inner product spaces}, Operator Theory: Advances
and Applications, 20, Birkh\"auser Verlag, Basel, 1986..

\bibitem{C} J. A. Clarkson, \textit{Uniformly convex spaces}, Trans. Amer. Math. Soc. \textbf{40} (1936), 396--414.

\bibitem{DFM} F. Dadipour, M. S. Moslehian, \textit{A characterization of inner product spaces related to the $p$-angular distance}, J. Math. Anal. Appl. \textbf{371} (2010), no. 2, 677--681.

\bibitem{DSS} F. Dadipour, F. Sadeghi, A. Salemi, \textit{Characterizations of inner product spaces involving homogeneity of isosceles orthogonality}. Arch. Math. (Basel) \textbf{104} (2015), no. 5, 431--439.

\bibitem{H} H. Dehghan, \textit{A characterization of inner product spaces related to the skew-angular distance}, Math. Notes \textbf{93} (2013), no. 4, 556--560.

\bibitem{D5} S. S. Dragomir, \textit{ A survey of some recent inequalities for the norm and numerical radius of operators in Hilbert spaces}, Banach J. Math. Anal. \textbf{1} (2007), no. 2, 154--175.

\bibitem{D2} S. S. Dragomir, \textit{New inequalities for the p-Angular distance in normed spaces with applications}, Ukrainian Math. J. \textbf{67} (2015), no. 1, 19--32.

\bibitem{D1} S. S. Dragomir, \textit{Upper and lower bounds for the $p$-angular distance in normed spaces with applications}, J. Math. Inequal. \textbf{8} (2014), no. 4, 947--961.

\bibitem{D9} C. F. Dunkl, K.S. Williams, Mathematical notes: A simple norm inequality, Amer. Math. Monthly \textbf{71} (1964), no. 1, 53--54.

\bibitem{F} F. A. Ficken, \textit{Note on the existence of scalar products in normed linear spaces}, Ann. of Math. (2) \textbf{45} (1944), 362--366.

\bibitem{VI} V. I. Gurari\u{l}, \textit{Strengthening the Dunkl-Williams inequality on the norm of elements of Banach spaces}, Dopov\=\i d\=\i\ Akad. Nauk Ukra\"\i n. RSR (1966), 35--38 (Ukrainian).

\bibitem{kh} E. Hewitt, K. Stromberg, \textit{Real and abstract analysis}, Springer, New York, 1965.

\bibitem{khu} G. N. Hile, \textit{Entire solutions of linear elliptic equations with Laplacian principal part}. Pacific J. Math. \textbf{62} (1976), no. 1, 127--140.

\bibitem{lorch} E. R. Lorch, \textit{On certain implications which characterize Hilbert space}, Ann. of Math. \textbf{49} (1948), no. 3, 523--532.

\bibitem{M} L. Maligranda, \textit{Simple norm inequalities}, Amer. Math. Monthly. \textbf{113} (2006), no. 3, 256--260.

\bibitem{MSM} M. S. Moslehian, \textit{Matrix Hermite--Hadamard type inequalities}, Houston J. Math. \textbf{39} (2013), no. 1, 177--189.
\end{thebibliography}

\end{document}